\documentclass[english,letterpaper,11pt,reqno]{amsart}

\usepackage{appendix}
\usepackage{amsbsy}
\usepackage{amsfonts}
\usepackage{amsmath}
\usepackage{amssymb}
\usepackage{amsthm}
\usepackage{graphicx}
\usepackage{ifthen}
\usepackage{textcomp}
\usepackage{soul}
\usepackage{enumitem,kantlipsum}
\usepackage{dsfont}
\usepackage[bookmarksnumbered,colorlinks]{hyperref}
\hypersetup{colorlinks=true, linkcolor=blue, citecolor=blue}
\newcommand{\lra}{\longrightarrow}

\newcommand{\RR}{\mathbb{R}}

\newcommand{\vep}{\varepsilon}

\makeatletter
\newcommand*{\defeq}{\mathrel{\rlap{%
                     \raisebox{0.3ex}{$\m@th\cdot$}}%
                     \raisebox{-0.3ex}{$\m@th\cdot$}}%
                     =}
\makeatother

\newtheorem{prop}{Proposition}
\newtheorem*{definition-non}{Definition}
\newtheorem*{theorem-non}{Theorem}
\newtheorem*{proposition-non}{Proposition}
\newtheorem*{lemma-non}{Lemma}
\newtheorem*{corollary-non}{Corollary}

\newcommand{\beqa}{\begin{eqnarray}}
\newcommand{\beq}{\begin{equation}}
\newcommand{\eeqa}{\end{eqnarray}}
\newcommand{\eeq}{\end{equation}}

\newcommand\ip[2]{g({#1},{#2})} %bracket scalar product
\newcommand\vv[1]{{\boldsymbol {\it #1}}} %bold vector
\newcommand\ii{i}

\newcommand\imp{\hspace{.2in}\Rightarrow\hspace{.2in}}

\newcommand\mb{\overline{\boldsymbol m}}
\newcommand\mm{{\boldsymbol m}}
\newcommand\kk{T}
\newcommand\xx{X}
\newcommand\yy{Y}
\newcommand\cd[2]{\nabla_{\!#1}{#2}}
\newcommand\cds[2]{\nabla^{\scriptscriptstyle L}_{\!#1}{#2}}
\newcommand\cdo[2]{\nabla^{\scriptscriptstyle o}_{\!#1}{#2}}

\newcommand\comma{\hspace{.2in},\hspace{.2in}}

\usepackage{enumitem, color, amssymb}
\begin{document}
\title[]{On the Einstein condition for Lorentzian 3-manifolds}
\author[]{Amir Babak Aazami}
\address{Clark University\hfill\break\indent
Worcester, MA 01610}
\email{Aaazami@clarku.edu}

%%%%%%%%%%%
\maketitle
\begin{abstract}
It is well known that in Lorentzian geometry there are no compact spherical space forms; in dimension 3, this means there are no closed Einstein 3-manifolds with positive Einstein constant.  We generalize this fact here, by proving that there are also no closed Lorentzian 3-manifolds $(M,g)$ whose Ricci tensor satisfies
$$
\text{Ric} = fg+(f-\lambda)T^{\flat}\otimes T^{\flat},
$$
for any unit timelike vector field $T$, any positive constant $\lambda$, and any function $f$ that never takes the values $0,\lambda$. (Observe that this reduces to the positive Einstein case when $f = \lambda$.)  We show that there is no such obstruction if $\lambda$ is negative. Finally, the ``borderline" case $\lambda = 0$ is also examined: we show that if $\lambda = 0$, then $(M,g)$ must be isometric to $(\mathbb{S}^1\!\times \!N,-dt^2\oplus h)$ with $(N,h)$ a Riemannian manifold.
\end{abstract}

%%%%%%%%%%%
\section{Introduction}
\label{sec:Intro}
The goal of this paper is to generalize the nonexistence of positive Lorentzian Einstein metrics in dimension 3, by proving the following
\begin{theorem-non}
\label{thm:1}
Let $M$ be a closed 3-manifold, $\lambda > 0$ a constant, and $f$ a smooth function that never takes the values $0,\lambda$.  Then there is no Lorentzian metric $g$ on $M$ whose Ricci tensor satisfies
\beqa
\label{eqn:1}
\emph{\text{Ric}} = fg+(f-\lambda)T^{\flat}\otimes T^{\flat},
\eeqa
for any unit timelike vector field $T$. If \eqref{eqn:1} holds with $\lambda = 0$, then $(M,g)$ is isometric to $(\mathbb{S}^1\!\times N,-dt^2\oplus h)$ with $(N,h)$ a Riemannian manifold.
\end{theorem-non}

The following remarks help to set this result in context:
\vskip 6pt
\begin{enumerate}[leftmargin=*]

\item[1.] Although our Theorem does not cover the case $f = \lambda$, nevertheless this would reduce \eqref{eqn:1} to the Einstein condition $\text{Ric} = \lambda g$ with positive Einstein constant.  In dimension 3, this is equivalent to $(M,g)$ having constant positive sectional curvature\,---\,and it is well known that there are \emph{no} such (closed) Lorentzian manifolds.   (Its proof came in two stages: \cite{CM} showed that there are no geodesically complete such manifolds, in any dimension; \cite{Klingler} then showed that every closed constant curvature Lorentzian manifold must be geodesically complete, with the flat case having already been established in \cite{carriere}.  See \cite{lundberg} for a comprehensive account.) It is in this sense that \eqref{eqn:1} generalizes the nonexistence of positive Einstein metrics; but more than that, it also shows that certain so called ``quasi-Einstein" metrics are also impossible.  Recall that a (semi-Riemannian) metric $g$ is \emph{quasi-Einstein} if there is a vector field $X$, a constant $\mu$, and a constant $m > 0$ such that
\beqa
\label{eqn:qEin}
\text{Ric} = \mu g - \frac{1}{2}\mathfrak{L}_Xg + \frac{1}{m}X^{\flat}\otimes X^{\flat};
\eeqa
see, e.g., \cite{limoncu}. Now observe that if $X$ is a unit timelike Killing vector field (i.e., $\mathfrak{L}_Xg = 0$), then \eqref{eqn:1} provides global obstructions to the existence of certain quasi-Einstein metrics \eqref{eqn:qEin}.

\item[2.] The condition imposed on $\lambda$ deserves commentary.  When $\lambda > 0$, then the endomorphism of the normal bundle $T^{\perp}$ given by \eqref{eqn:DT} below will have, presuming \eqref{eqn:1} holds, \emph{real} eigenvalues, and this leads directly to the existence of a distinguished frame on $(M,g)$ that is crucial to our proof.  This will not be the case if $\lambda < 0$\,---\,nor should this be surprising: there are in fact many examples of Lorentzian Einstein metrics with \emph{negative} Einstein constant (see, e.g., \cite{kulkarni} and \cite{Goldman}), so that a result like \eqref{eqn:1} should not be expected when $\lambda < 0$. Indeed, as we show in Section \ref{sec:NP} below, the Lorentzian metric $g_{\scriptscriptstyle L}$ on $\mathbb{S}^3$ defined via the standard (round) Riemannian metric $\mathring{g}$ by
\beqa
\label{eqn:LS3}
g_{\scriptscriptstyle L} \defeq \mathring{g} - 2\mathring{g}(T,\cdot) \otimes \mathring{g}(T,\cdot),
\eeqa
where $T$ is the Hopf Killing vector field, satisfies \eqref{eqn:1} with $f = 8, \lambda = -2$:
\beqa
\label{eqn:ricciS3}
\text{Ric}_{\scriptscriptstyle L} = 8g_{\scriptscriptstyle L} + (8-(-2))T^\flat \otimes T^\flat.
\eeqa
(Here $T^\flat = g_{\scriptscriptstyle L}(T,\cdot) = -\mathring{g}(T,\cdot)$.) Finally, the nowhere vanishing condition on $f$ arises for the following reason: our Theorem is obtained by showing that if \eqref{eqn:1} holds, then there necessarily exists a vector field $Z$ on $M$ and a smooth function $\psi$ such that $Z(\psi) = f$, which is impossible on a compact manifold when $f$ is nowhere vanishing\,---\,or else has just one zero.

\item[3.] As all compact 3-manifolds have Euler characteristic zero, there is no topological obstruction to $M$ admitting a Lorentzian metric (see, e.g., \cite{o1983}). Furthermore, \eqref{eqn:LS3} is but one instance of the following fact: every Lorentzian metric $g$ with a unit timelike vector field $T$ is necessarily of the form
$$
g = g_{\scriptscriptstyle R} - 2g_{\scriptscriptstyle R}(T,\cdot)\otimes g_{\scriptscriptstyle R}(T,\cdot)
$$
for some Riemannian metric $g_{\scriptscriptstyle R}$ on $M$ with $g_{\scriptscriptstyle R}(T,T) = 1$. In other words, there are many candidates for $g$ and $T$ on any closed 3-manifold $M$; nevertheless, by our Theorem, they will all fail to satisfy \eqref{eqn:1}.

\item[4.] Our proof uses the three-dimensional version of the \emph{Newman-Penrose formalism} \cite{newpen62}, which we outline in Section \ref{sec:NP}. This is a frame technique which has by now appeared in many guises in dimensions 3 and 4, both  Lorentzian and Riemannian; see, e.g., \cite{hall87,schmidt2014,NTC,AMB,bettiol2018}.  Indeed, our proof, which appears in Section \ref{sec:proof}, is a Lorentzian conversion of two Riemannian results in \cite{AMB}; Section \ref{sec:proof} will make their precise relationship clear.  The principal virtue of the Newman-Penrose formalism is that it converts second-order differential equations involving curvature into first-order differential equations involving properties of the flow of a privileged vector field, in our case the vector field $T$ in \eqref{eqn:1}. Doing so can, in certain settings, simplify the analysis considerably\,---\,our Theorem, it turns out, is precisely such a setting.
\item[5.] Observe that if \eqref{eqn:1} were to hold, then $\lambda,f,f$ would be the eigenvalues of the Ricci operator $\widehat{\text{Ric}}\colon TM \lra TM$, the smooth bundle endomorphism whose action $X \mapsto \widehat{\text{Ric}}(X)$ is the unique vector satisfying
$$
g(\widehat{\text{Ric}}(X),Y) = \text{Ric}(X,Y) \hspace{.2in}\text{for all}~Y \in TM.
$$
Much work has been done in \emph{locally} classifying Lorentzian 3-manifolds (i.e., on $\RR^3$) according to the possible eigenvalues of $\widehat{\text{Ric}}$, in particular when these eigenvalues are constants; see \cite{bueken2,bueken,calvaruso}.  Although we make no use of $\widehat{\text{Ric}}$ explicitly\,---\,in the Lorentzian setting, $\widehat{\text{Ric}}$ is not, in general, diagonalizable\,---\,nevertheless it provides another way to appreciate the generalization achieved in our Theorem, since the Einstein case $\text{Ric} = \lambda g$ corresponds to the eigenvalues $\lambda,\lambda,\lambda$.
\end{enumerate}

%%%%%%%%%%%
\section{The Newman-Penrose formalism for Lorentzian 3-manifolds}
\label{sec:NP}
Let $\kk$ be a smooth unit length timelike vector field defined in an open subset of a Lorentzian 3-manifold $(M,g)$ without boundary, so that $\cd{v}{\kk} \perp \kk$ for all vectors $v$ ($\nabla$ is the Levi-Civita connection, and we adopt the index $-++$ for the Lorentzian metric $g$).  Let $\xx$ and $\yy$ be two smooth spacelike vector fields such that $\{\kk,\xx,\yy\}$ is a local orthonormal frame, where by ``timelike" and ``spacelike" we mean simply that
$$
g(T,T) = -1 \comma g(X,X) = g(Y,Y) = 1.
$$
For such a $\kk$, there is an endomorphism $D$ defined on the normal bundle $\kk^{\perp} \subset TM$:
\beqa
\label{eqn:DT}
D\colon \kk^{\perp} \lra \kk^{\perp}\hspace{.2in},\hspace{.2in}v\ \mapsto\ \cd{v}{\kk}.
\eeqa
As is well known, the matrix of $D$ with respect to the frame $\{\kk,\xx,\yy\}$,
\beqa
\label{eqn:matrix2}
D = \begin{pmatrix}
        \ip{\cd{\xx}{\kk}}{\xx} & \ip{\cd{\yy}{\kk}}{\xx}\\
        \ip{\cd{\xx}{\kk}}{\yy} & \ip{\cd{\yy}{\kk}}{\yy}
    \end{pmatrix}\nonumber
\eeqa
carries three crucial pieces of information associated to the flow of $\kk$:
\begin{enumerate}[leftmargin=*]
\item[1.] The divergence of $\kk$, denoted $\text{div}\,\kk$, is the trace of $D$.
\item[2.] By Frobenius's theorem, $\kk^{\perp}$ is integrable if and only if the off-diagonal elements of \eqref{eqn:matrix} satisfy
\beqa
\label{eqn:rotation}
\omega \defeq \ip{\kk}{[\xx,\yy]} = \ip{\cd{\yy}{\kk}}{\xx} - \ip{\cd{\xx}{\kk}}{\yy} = 0.
\eeqa
Since $\omega^2$ equals the determinant of the anti-symmetric part of $D$ (see \eqref{eqn:matrix} below), $\omega$ is invariant up to sign, called the \emph{twist} of $\kk$.  We say that the flow of $\kk$ is \emph{twist-free} if $\omega = 0$.
\item[3.] The third piece of information is the \emph{shear} $\sigma$ of $\kk$; it is given by the trace-free symmetric part of $D$, whose components $\sigma_1,\sigma_2$ we combine here into a complex-valued quantity, for reasons that will become clear below:
\beqa
\hspace{.26in}\sigma\!\!&\defeq&\!\! \underbrace{\,\frac{1}{2}\Big(\ip{\cd{\yy}{\kk}}{\yy} - \ip{\cd{\xx}{\kk}}{\xx}\Big)\,}_{\sigma_1} +\ \ii\underbrace{\,\frac{1}{2}\Big(\ip{\cd{\yy}{\kk}}{\xx} + \ip{\cd{\xx}{\kk}}{\yy}\Big)\,}_{\sigma_2}.\nonumber\\
&&\label{eqn:shear}
\eeqa
Although $\sigma$ itself is not invariant, its magnitude $|\sigma^2|$ is: by \eqref{eqn:matrix} below, it is minus the determinant of the trace-free symmetric of $D$.  We say that the flow of $\kk$ is \emph{shear-free} if $\sigma = 0$.  As with being twist-free, being shear-free is a frame-independent statement. In terms of $\text{div}\,\kk$, $\omega$, and $\sigma$, $D$ is
\beqa
\label{eqn:matrix}
D = \begin{pmatrix}
        \frac{1}{2}\text{div}\,\kk - \sigma_1 & \sigma_2 + \frac{\omega}{2}\\
        \sigma_2 - \frac{\omega}{2} & \frac{1}{2}\text{div}\,\kk + \sigma_1
    \end{pmatrix}\cdot
\eeqa
Note that \eqref{eqn:DT} has been applied to arbitrary $n$-dimensional Lorentzian manifolds, yielding integral inequalities via Bochner's technique, in particular when $D$ is skew-symmetric, or (in dimension 3) when the vector field of interest is spacelike, in \cite{sanchez_bochner} and \cite{sanchez_bochner2}.
\end{enumerate}

%%%%%%%%%%%
We now present the Newman-Penrose formalism for Lorentzian 3-manifolds, which is well known; see \cite{hall87}, and more recently, \cite{NTC}. Our presentation here parallels the three-dimensional Riemannian treatment to be found in \cite{AMB}, and is meant to fix notation; in what follows, any sign changes that arise due to the Lorentzian index, as compared to the Riemannian case in \cite{AMB}, are indicated by \textcolor{red}{red} text.  Let $\{\kk,\xx,\yy\}$ be as above and define the complex-valued quantities
\beqa
\label{eqn:complex}
\mm \defeq \frac{1}{\sqrt{2}}(\xx-\ii \yy)\hspace{.2in},\hspace{.2in}\mb \defeq \frac{1}{\sqrt{2}}(\xx+\ii \yy).\nonumber
\eeqa
Henceforth we work with the complex frame $\{\kk,\mm,\mb\}$, for which only $\ip{\kk}{\kk} = \textcolor{red}{-}1$ and $\ip{\mm}{\mb} = 1$ are nonzero.  The following quantities associated to this complex triad play a central role in all that follows.

\begin{definition-non}
\label{def:spin}
The {\rm spin coefficients} of the complex frame $\{\kk,\mm,\mb\}$ are the complex-valued functions
\beqa
\label{eqn:sc}
\kappa\!\!\! &=&\!\!\! -\ip{\cd{\kk}{\kk}}{\mm}\hspace{.2in},\hspace{.2in}\rho=-\ip{\cd{\mb}{\kk}}{\mm}\hspace{.2in},\hspace{.2in}\sigma=-\ip{\cd{\mm}{\kk}}{\mm},\nonumber\\
&&\hspace{.45in}\vep=\ip{\cd{\kk}{\mm}}{\mb}\hspace{.2in},\hspace{.2in}\beta=\ip{\cd{\mm}{\mm}}{\mb}.\nonumber
\eeqa
\end{definition-non}
Note that the flow of $\kk$ is geodesic, $\cd{\kk}{\kk} = 0$, if and only if $\kappa = 0$; that
$$
\vep = \ii \ip{\cd{\kk}{\xx}}{\yy}
$$
is purely imaginary; that $\sigma$ is actually the complex shear \eqref{eqn:shear}; and finally that the spin coefficient $\rho$ is given by
\beqa
\label{eqn:rho2}
-2\rho = \text{div}\,\kk + \ii\, \omega.
\eeqa
It is clear that $\kappa,\rho,\sigma$ directly represent the three geometric properties of the flow of $\kk$ discussed  above.  In terms of the five spin coefficients in Definition \ref{def:spin}, the covariant derivatives of $\{\kk,\mm,\mb\}$ are given by
\beqa
\cd{\kk}{\kk} \!\!&=&\!\! -\bar{\kappa}\,\mm - \kappa\,\mb,\nonumber\\
\cd{\mm}{\kk} \!\!&=&\!\! -\bar{\rho}\,\mm - \sigma\,\mb,\nonumber\\
\cd{\kk}{\mm} \!\!&=&\!\! \textcolor{red}{-}\kappa\,\kk + \vep\,\mm,\nonumber\\
\cd{\mm}{\mm} \!\!&=&\!\! \textcolor{red}{-}\sigma\,\kk + \beta\,\mm,\nonumber\\
\cd{\mm}{\mb} \!\!&=&\!\! \textcolor{red}{-}\bar{\rho}\,\kk - \beta\,\mb,\nonumber
\eeqa
while their Lie brackets are
\beqa
[\kk,\mm] \!\!&=&\!\! \textcolor{red}{-}\kappa\,\kk + (\vep + \bar{\rho})\,\mm + \sigma\,\mb,\nonumber\\ 
\,[\mm,\mb] \!\!&=&\!\! \textcolor{red}{-}(\bar{\rho} - \rho)\,\kk + \bar{\beta}\,\mm - \beta\,\mb.\nonumber
\eeqa
All other covariant derivatives, as well as the Lie bracket $[\kk,\mb]$, are obtained by complex conjugation.  Now onto curvature; begin by observing that the Riemann and Ricci tensors
\beqa
R(\vv{u},\vv{v},\vv{w},\vv{z}) \!\!&=&\!\! \ip{\cd{\vv{u}}{ \cd{\vv{v}}{\vv{w}}} - \cd{\vv{v}}{ \cd{\vv{u}}{\vv{w}}} - \cd{[\vv{u},\vv{v}]}{\vv{w}}}{\vv{z}},\nonumber\\
\text{Ric}(\cdot\,,\cdot) \!\!&=&\!\! \textcolor{red}{-}R(\kk,\cdot,\cdot,\kk) + R(\mm,\cdot,\cdot,\mb) + R(\mb,\cdot,\cdot,\mm),\nonumber
\eeqa
satisfy the following relationships in the complex frame $\{\kk,\mm,\mb\}$:
$$
\left\{
\begin{array}{rcl}
\text{Ric}(\mm,\mm) \!\!&=&\!\! \textcolor{red}{+}R(\kk,\mm,\kk,\mm),\nonumber\\
\text{Ric}(\kk,\kk) \!\!&=&\!\! -2R(\kk,\mm,\kk,\mb),\nonumber\\
\text{Ric}(\kk,\mm) \!\!&=&\!\! -R(\kk,\mm,\mm,\mb),\nonumber\\
\text{Ric}(\mm,\mb) \!\!&=&\!\! \textcolor{red}{-}\frac{1}{2}\text{Ric}(\kk,\kk) - R(\mb,\mm,\mm,\mb).\nonumber
\end{array}
\right.
$$
Using these, and expressing the Riemann tensor as
\beqa
R(\vv{u},\vv{v},\vv{w},\vv{z}) \!\!&=&\!\! \vv{u}\ip{\cd{\vv{v}}{\vv{w}}}{\vv{z}} - \ip{\cd{\vv{v}}{\vv{w}}}{\cd{\vv{u}}{\vv{z}}} - \vv{v}\ip{\cd{\vv{u}}{\vv{w}}}{\vv{z}}\nonumber\\
&& \hspace{1in} +\ \ip{\cd{\vv{u}}{\vv{w}}}{\cd{\vv{v}}{\vv{z}}} - \ip{\cd{[\vv{u},\vv{v}]}{\vv{w}}}{\vv{z}},\nonumber
\eeqa
leads to the following equations; along with \eqref{eqn:bid} and \eqref{eqn:bid2} below, play the crucial role in the Newman-Penrose formalism:
\beqa
\label{eqn:Sachs1}
\kk(\rho) - \mb(\kappa) \!\!&=&\!\! \textcolor{red}{-}|\kappa|^2 + |\sigma|^2 + \rho^2 + \kappa\bar{\beta} + \frac{1}{2} {\rm Ric}(\kk,\kk),\label{eqn:S1}\\
\kk(\sigma) - \mm(\kappa) \!\!&=&\!\! \textcolor{red}{-}\kappa^2 + 2\sigma\vep + \sigma(\rho + \bar{\rho}) - \kappa \beta \,\textcolor{red}{-}\, {\rm Ric}(\mm,\mm),\phantom{\frac{1}{2}}\label{eqn:S2}\\
\mm(\rho) - \mb(\sigma) \!\!&=&\!\! 2 \sigma\bar{\beta} \,\textcolor{red}{-}\, (\bar{\rho}-\rho)\kappa + {\rm Ric}(\kk,\mm),\phantom{\frac{1}{2}}\label{eqn:S3}\\
\kk(\beta) - \mm(\vep) \!\!&=&\!\! \sigma(\bar{\kappa} - \bar{\beta}) + \kappa (\textcolor{red}{-}\vep - \bar{\rho}) + \beta(\vep + \bar{\rho}) - {\rm Ric}(\kk,\mm),\phantom{\frac{1}{2}}\nonumber\\
\mm(\bar{\beta}) + \mb(\beta) \!\!&=&\!\! \textcolor{red}{-}\,|\sigma|^2\,\textcolor{red}{+}\, |\rho|^2 -2|\beta|^2 \,\textcolor{red}{-}\, (\rho - \bar{\rho})\vep - {\rm Ric}(\mm,\mb) \,\textcolor{red}{-}\, \frac{1}{2} {\rm Ric}(\kk,\kk).\nonumber\\\label{eqn:S4}
\eeqa

Finally, there are two nontrivial differential Bianchi identities,
\beqa
(\cd{\kk}{R})(\kk,\mm,\mm,\mb) + (\cd{\mm}{R})(\kk,\mm,\mb,\kk) + (\cd{\mb}{R})(\kk,\mm,\kk,\mm)=0,\nonumber\\
(\cd{\kk}{R})(\mb,\mm,\mm,\mb) + (\cd{\mm}{R})(\mb,\mm,\mb,\kk) + (\cd{\mb}{R})(\mb,\mm,\kk,\mm)=0,\nonumber
\eeqa
which, in terms of spin coefficients, take the following forms:
\beqa
&&\hspace{-.4in}\kk({\rm Ric}(\kk,\mm))\, -\, \frac{1}{2}\mm({\rm Ric}(\kk,\kk))\ \textcolor{red}{-}\ \mb({\rm Ric}(\mm,\mm))=\label{eqn:bid}\\
&&\hspace{-.1in}-\,\kappa\,\big(\textcolor{red}{+}{\rm Ric}(\kk,\kk) + \text{Ric}(\mm,\mb)\big) + \big(\vep + 2\rho + \bar{\rho}\big){\rm Ric}(\kk,\mm)\phantom{\frac{1}{2}}\nonumber\\
&&\hspace{1.3in}+\, \sigma\,{\rm Ric}(\kk,\mb)- \big(\bar{\kappa} \,\textcolor{red}{-}\, 2\bar{\beta}\big){\rm Ric}(\mm,\mm)\phantom{\frac{1}{2}}\nonumber\\
\text{and}&&\nonumber\\
&&\hspace{-.4in}\mm({\rm Ric}(\kk,\mb)) + \mb({\rm Ric}(\kk,\mm)) - \kk\big({\rm Ric}(\mm,\mb) \,\textcolor{red}{+}\, (1/2){\rm Ric}(\kk,\kk)\big) =\phantom{\frac{1}{2}}\nonumber\\
&&\hspace{-.1in}\textcolor{red}{-}\,(\rho+\bar{\rho})\big({\rm Ric}(\kk,\kk) \,\textcolor{red}{+}\, {\rm Ric}(\mm,\mb)\big) - \bar{\sigma}{\rm Ric}(\mm,\mm) - \sigma{\rm Ric}(\mb,\mb)\phantom{\frac{1}{2}}\label{eqn:bid2}\\
&&\hspace{1.2in} -\, \big(\textcolor{red}{-}2\bar{\kappa} + \bar{\beta}\big){\rm Ric}(\kk,\mm) - \big(\textcolor{red}{-}2\kappa + \beta\big){\rm Ric}(\kk,\mb).\phantom{\frac{1}{2}}\nonumber
\eeqa

E.g., to derive \eqref{eqn:bid2}, expand the second differential Bianchi identity, beginning with its first term:
\beqa
&&\hspace{-.1in}(\cd{\kk}{R})(\mb,\mm,\mm,\mb)=\kk(R(\mb,\mm,\mm,\mb))\,-\, R(\cd{\kk}{\mb},\mm,\mm,\mb)\nonumber\\
&&\hspace{.1in}\,-\, R(\mb,\cd{\kk}{\mm},\mm,\mb)\,-\, R(\mb,\mm,\cd{\kk}{\mm},\mb)\,-\, R(\mb,\mm,\mm,\cd{\kk}{\mb}).\nonumber
\eeqa
In terms of spin coefficients and the Ricci tensor, each term is
\beqa
\kk(R(\mb,\mm,\mm,\mb)) \!\!&=&\!\! -\kk({\rm Ric}(\mm,\mb)\,\textcolor{red}{+}\, \frac{1}{2}{\rm Ric}(\kk,\kk)),\phantom{\underbrace{R}_{0}}\nonumber\\
R(\cd{\kk}{\mb},\mm,\mm,\mb) \!\!&=&\!\! \textcolor{red}{-}\bar{\kappa}\,\underbrace{R(\kk,\mm,\mm,\mb)}_{-\text{Ric}(\kk,\mm)}\,+\ \bar{\vep}\!\!\!\!\underbrace{R(\mb,\mm,\mm,\mb)}_{-\text{Ric}(\mm,\mb)\,\textcolor{red}{-}\,\frac{1}{2}\text{Ric}(\kk,\kk)},\nonumber\\ 
&=&\!\! R(\mb,\mm,\mm,\cd{\kk}{\mb}),\phantom{\underbrace{R}_{0}}\nonumber\\
R(\mb,\cd{\kk}{\mm},\mm,\mb) \!\!&=&\!\! \textcolor{red}{-}\kappa\,\underbrace{R(\mb,\kk,\mm,\mb)}_{-\text{Ric}(\kk,\mb)}\ +\ \vep\!\!\!\!\underbrace{R(\mb,\mm,\mm,\mb)}_{-\text{Ric}(\mm,\mb)\,\textcolor{red}{-}\,\frac{1}{2}\text{Ric}(\kk,\kk)},\nonumber\\
 &=&\!\! R(\mb,\mm,\cd{\kk}{\mm},\mb),\phantom{\underbrace{R}_{0}}\nonumber
\eeqa
Thus the term $(\cd{\kk}{R})(\mb,\mm,\mm,\mb)$ simplifies to
$$
-\kk\big({\rm Ric}(\mm,\mb)\,\textcolor{red}{+}\, (1/2){\rm Ric}(\kk,\kk)\big)\, \textcolor{red}{-}\, 2\kappa \text{Ric}(T,\mb)\,\textcolor{red}{-}\,2\bar{\kappa} \text{Ric}(T,\mm), 
$$
where two terms cancel because $\vep + \bar{\vep} = 0$.  Repeating this process on the remaining terms in the second differential Bianchi identity yields \eqref{eqn:bid2}; the first differential Bianchi identity \eqref{eqn:bid} is similarly derived.  This concludes the derivation of the Newman-Penrose formalism for Lorentzian 3-manifolds.  As a first, minor application, let us verify that

\begin{lemma-non}
On $\mathbb{S}^3$, the Lorentzian metric $g_{\scriptscriptstyle L}$ given by \eqref{eqn:LS3} satisfies \eqref{eqn:ricciS3}.
\end{lemma-non}

\begin{proof}
Let $\mathring{g}$ denote the standard (round) metric on $\mathbb{S}^3$.  Then on $\RR^4 = \{(x^1,y^1,x^2,y^2)\}$, the restriction to $(\mathbb{S}^3,\mathring{g})$ of the vector field
$$
T = \sum_{i=1}^2 (-y^i\partial_{x^i} + x^i\partial_{y^i})
$$
is a unit-length Killing vector field, tangent to the Hopf fibration. With respect to the Lorentzian metric $g_{\scriptscriptstyle L}$ in \eqref{eqn:LS3},
\beqa
\label{eqn:LS3*}
g_{\scriptscriptstyle L} \defeq \mathring{g} - 2\mathring{g}(T,\cdot) \otimes \mathring{g}(T,\cdot),
\eeqa
the Koszul formula shows that $T$ remains a Killing vector field of unit length, but now a timelike one: $g_{\scriptscriptstyle L}(T,T) = -1$. As a consequence, and using the Koszul formula again, any local $g_{\scriptscriptstyle L}$-orthonormal frame $\{T,X,Y\}$ will have covariant derivatives
\beqa
\cds{X}{T} = -\cdo{X}{T} &\comma& \cds{Y}{T} = -\cdo{Y}{T},\nonumber\\
\cds{T}{X} = \cdo{T}{X} - 2\cdo{X}{T} &\comma& \cds{T}{Y} = \cdo{T}{Y} - 2\cdo{Y}{T},\nonumber\\
\cds{a}{b} = \cdo{a}{b}&&\hspace{-.5in}\text{for all other}~a,b \in \{T,X,Y\},\nonumber
\eeqa
where $\nabla^{\scriptscriptstyle L}$ is the Levi-Civita connection of $g_{\scriptscriptstyle L}$ and $\nabla^{\scriptscriptstyle o}$ that of $\mathring{g}$ (for more general formulae for Lorentzian metrics of the form \eqref{eqn:LS3*}, consult \cite{olea}). It now remains to compute the Ricci tensor, but in fact \eqref{eqn:S1}-\eqref{eqn:S3} will simplify the computation.  This is because any unit timelike Killing vector field $T$ on a Lorentzian 3-manifold will satisfy
$$
\kappa = \rho+\bar{\rho} = \sigma = 0
$$
with respect to any complex frame $\{T,\mm,\mb\}$ (in fact in dimension 3, such a vector field is completely determined by these equations). Inserting these into \eqref{eqn:S1}, its real part simplifies to
\beqa
\label{eqn:bochner}
\text{Ric}_{\scriptscriptstyle L}(T,T) = \frac{\omega_{\scriptscriptstyle L}^2}{2},
\eeqa
where the twist $\omega_{\scriptscriptstyle L}^2$ satisfies
$$
\omega_{\scriptscriptstyle L}^2 \overset{\eqref{eqn:rotation}}{=} g_{\scriptscriptstyle L}(T,[X,Y])^2 \overset{\eqref{eqn:LS3*}}{=} \mathring{g}(T,[X,Y])^2 =: \mathring{\omega}^2 = 4.
$$
(That $\mathring{\omega}^2 = 4$ follows from the Einstein condition $\text{Ric}_{\mathring{g}} = 2\mathring{g}$, together with the Riemannian version of \eqref{eqn:S1}, obtained from \eqref{eqn:S1} by changing each (red) ``$-$" to ``$+$"; doing so leads to $\text{Ric}_{\mathring{g}}(T,T) = \frac{\mathring{\omega}^2}{2}$ as in \eqref{eqn:bochner}.)  Next, \eqref{eqn:S2} and \eqref{eqn:S3} simplify, respectively, to
$$
\text{Ric}_{\scriptscriptstyle L}(\mm,\mm) =  0 \comma \text{Ric}_{\scriptscriptstyle L}(T,\mm) = 0,
$$
where in the latter equality we've used the  fact that $\mm(\rho) = 0$  because $\omega_{\scriptscriptstyle L}^2 = 4$ is a constant. All in all, we thus have that
$$
\text{Ric}_{\scriptscriptstyle L}(T,T) = 2 \comma \text{Ric}_{\scriptscriptstyle L}(X,X) = \text{Ric}_{\scriptscriptstyle L}(Y,Y),
$$
and $\text{Ric}_{\scriptscriptstyle L}(X,Y) = \text{Ric}_{\scriptscriptstyle L}(T,X) = \text{Ric}_{\scriptscriptstyle L}(T,Y) = 0$. The final step in the proof is to show that $\text{Ric}_{\scriptscriptstyle L}(X,X) = \text{Ric}_{\scriptscriptstyle L}(X,X)$ is in fact a global constant, which indeed it is, after one expands
$$
\text{Ric}_{\scriptscriptstyle L}(X,X) =  -\underbrace{\,R_{\scriptscriptstyle L}(T,X,X,T)\,}_{-1} + \underbrace{\,R_{\scriptscriptstyle L}(Y,X,X,Y)\,}_{7} = 8,
$$
where the computations of the curvature tensor components are straightforwardly carried out, using both the Einstein condition $\text{Ric}_{\mathring{g}} = 2\mathring{g}$ and the fact that $T$ is a unit Killing vector field in both $(\mathbb{S}^3,\mathring{g})$ and $(\mathbb{S}^3,g_{\scriptscriptstyle L})$. That $g_{\scriptscriptstyle L}$ satisfies \eqref{eqn:ricciS3} now follows.
\end{proof}

%%%%%%%%%%%
\section{Evolution equations for divergence, twist, and shear}
\label{sec:proof0}
We first need to gather some information regarding the flow of $T$; the first-order differential equations appearing here are, essentially, ``what the Newman-Penrose formalism is good for."
\begin{prop}
\label{prop:1}
Let $(M,g)$ be a Lorentzian 3-manifold whose Ricci tensor satisfies
$$
\emph{\text{Ric}} = fg+(f-\mu)T^{\flat}\otimes T^{\flat},
$$
for some unit timelike vector field $T$, constant $\mu$, and smooth function $f$ which never takes the value $\mu$.  Then $T$ has geodesic flow, and its divergence, twist, and shear satisfy the following differential equations:
\beqa
T(\emph{\text{div}}\,T) \!\!&=&\!\! \frac{\omega^2}{2} - 2|\sigma|^2 -\frac{1}{2}(\emph{\text{div}}\,T)^2 + \mu,\label{ODE:div}\\
T(\omega^2) \!\!&=&\!\! -2(\emph{\text{div}}\,T)\,\omega^2,\label{ODE:twist}\\
T(|\sigma|^2) \!\!&=&\!\! -2(\emph{\text{div}}\,T)\,|\sigma|^2.\label{ODE:shear}
\eeqa
Furthermore, $f$ satisfies
\beqa
\label{ODE:h}
T(f-\mu) = -(\emph{\text{div}}\,T)(f-\mu),
\eeqa
and, recalling \eqref{eqn:matrix}, the function $H \defeq \emph{\text{det}}\,D - \frac{\mu}{2}$ satisfies
\beqa
T(\emph{\text{div}}\,T) \!\!&=&\!\! 2H - (\emph{\text{div}}\,T)^2 + 2\mu,\label{ODE:H1}\\
T(H) \!\!&=&\!\! -(\emph{\text{div}}\,T)H.\label{ODE:H2}
\eeqa
\end{prop}

\begin{proof}
Let $\{T,\xx,\yy\}$ be an orthonormal frame, with $\xx,\yy$ possibly only locally defined, and let $\{T,\mm,\mb\}$ be the corresponding complex frame (recall that $\text{div}\,T, \omega^2$, and $|\sigma|^2$ are globally defined, frame-independent quantities). Then the Ricci tensor in this complex frame satisfies
$$
\text{Ric}(T,T) = -\mu \comma \underbrace{\,\text{Ric}(\mm,\mb)\,}_{\frac{1}{2}(\text{Ric}(X,X)\,+\, \text{Ric}(Y,Y))} \!\!\!\!\!\!\!\!\!=\ f,
$$
with all other components vanishing.  That $T$ has geodesic flow, $\cd{T}{T} = 0$, now follows from the differential Bianchi identity \eqref{eqn:bid}, which reduces to
$$
\kappa \underbrace{\,(\text{Ric}(\mm,\mb) + \text{Ric}(T,T))\,}_{f\,-\,\mu\,\neq\,0} = 0 \imp \kappa = 0.
$$
Since $\kappa = 0$ if and only if $\cd{T}{T} = 0$, this proves the geodesic flow of $T$. Next, \eqref{ODE:div} and \eqref{ODE:twist} are the real and imaginary parts of \eqref{eqn:S1}, respectively, after setting $\kappa = 0$ therein. With $\kappa = \text{Ric}(\mm,\mm) = 0$, \eqref{eqn:S2} also simplifies, to
$$
\kk(\sigma) = 2\sigma\vep + \sigma\!\!\underbrace{\,(\rho + \bar{\rho})\,}_{-\text{div}\,T},
$$
from which \eqref{ODE:shear} follows because $|\sigma|^2 = \sigma\bar{\sigma}$ and $\vep + \bar{\vep} = 0$. The second differential Bianchi identity \eqref{eqn:bid2} yields 
$$
-T(f-\mu/2) = (\text{div}\,T)(-\mu+f),
$$
which is \eqref{ODE:h}, since $T(f-\mu/2) = T(f-\mu)$. Finally, as
$$
\text{det}\,D = \frac{\omega^2}{4} - |\sigma|^2 + \frac{(\text{div}\,T)^2}{4},
$$
\eqref{ODE:H1} and \eqref{ODE:H2} both follow from \eqref{ODE:div}-\eqref{ODE:shear}.
\end{proof}

An immediate consequence of these evolution equations is the following

\begin{corollary-non}
\label{cor:1}
Assume the hypotheses of Proposition \ref{prop:1}.  If $\mu \geq 0$ and $M$ is closed, then $T$ is also divergence-free.
\end{corollary-non}

\begin{proof}
By Proposition \ref{prop:1}, $T$ has geodesic flow; because $M$ is closed, this flow is complete.  We now consider the cases $\mu > 0$ and $\mu = 0$ separately:
\begin{enumerate}[leftmargin=*]

\item[1.] $\mu > 0$: To show that $\text{div}\,T = 0$, we will use \eqref{ODE:H2}, by showing that $H$ is in fact a nonzero constant on $M$.  Indeed, suppose that $H(p) = 0$ at some point $p \in M$, and let $\gamma^{(p)}(t)$ be the (complete) integral curve of $T$ starting at $p$.  By \eqref{ODE:H2}, the function $H \circ \gamma^{(p)}\colon \RR \lra \RR$ is identically zero; by \eqref{ODE:H1}, $\theta(t) \defeq (\text{div}\,T\circ \gamma^{(p)})(t)$ satisfies
\beqa
\label{eqn:divmu}
\frac{d\theta}{dt} = -\theta^2 + 2\mu.
\eeqa
With $\mu > 0$, this has complete solutions $\theta(t) = \sqrt{2\mu}\,\text{tanh}\big(\sqrt{2\mu}\,t+c\big)$ and $\theta(t) = \pm \sqrt{2\mu}$. But in fact all of these solutions are impermissible, as can be seen by restricting \eqref{ODE:h} to $\gamma^{(p)}$. Indeed, substituting $\theta(t) = \pm\sqrt{2\mu}$ into \eqref{ODE:h} yields the solutions
$$
f(t)-\mu = (f(0)-\mu)e^{\mp\sqrt{2\mu}\,t}.
$$
Likewise, the solution $\theta(t) = \sqrt{2\mu}\,\text{tanh}\big(\sqrt{2\mu}\,t+c\big)$, when it is inserted into \eqref{ODE:h}, yields the solution
$$
f(t)-\mu = \big(f(0)-\mu\big)\text{sech} \big(\sqrt{2\mu}\,t+c\big).
$$
But in either case, the right-hand side goes to zero whereas the left-hand side is bounded away from zero (recall that $M$ is closed and $f$ never takes the value $\mu$). Thus $H$ must be nowhere vanishing on $M$, in which case, consider $1/H$:
\beqa
\label{eqn:Hmu}
T(1/H) \overset{\eqref{ODE:H2}}{=} \frac{\text{div}\,T}{H} \imp T(T(1/H)) \overset{\eqref{ODE:H1}}{=} 2 +\frac{2\mu}{H}\cdot
\eeqa
The latter equation has solution
\beqa
\label{eqn:HH1}
(1/H)(t) = -\frac{1}{\mu} + c_1e^{\sqrt{2\mu}\,t} + c_2 e^{-\sqrt{2\mu}\,t},\nonumber
\eeqa
but unless $c_1 = c_2 = 0$, this solution is unbounded.  We therefore conclude that the function $H$ is a nonzero constant on $M$, which immediately implies that $\text{div}\,T = 0$ by \eqref{ODE:H2}.

\item[2.] $\mu = 0$: Once again, suppose that $H(p)  = 0$, so that $H \circ \gamma^{(p)}\colon \RR \lra \RR$ is identically zero, in which case \eqref{ODE:H1} now becomes
$$
\frac{d\theta}{dt} = -\theta^2.
$$ 
This has $\theta(t) = 0$ as its only complete solution. If $H(p) \neq 0$, so that $H \circ \gamma^{(p)}\colon \RR \lra \RR$ is nowhere vanishing, then applying \eqref{eqn:Hmu} along $\gamma^{(p)}$ yields $(1/H)''(t) = 2$, hence
$$
(1/H)(t) = t^2+c_1t + c_2 \overset{\eqref{eqn:Hmu}}{\imp} \theta(t) = \frac{2t+c_1}{t^2+c_1t + c_2},
$$
with $c_1^2 < 4c_2^2$ to ensure that $(1/H)(t)$ is nowhere vanishing. But then \eqref{ODE:h} would yield
$$
f(t) = \frac{c_3}{t^2+c_1t + c_2},
$$
which is impossible because $f$, never taking the value $\mu = 0$, must be bounded away from $0$ on closed $M$. We conclude that $H$ must be the zero function, and so $\text{div}\,T = 0$ once again.
\end{enumerate}
This completes the proof.
\end{proof}

(As an aside, it is instructive to consider what happens when $\mu < 0$. In this case, if $H \circ \gamma^{(p)}\colon \RR \lra \RR$ is identically zero, then
$$
\frac{d\theta}{dt} = -\theta^2 + 2\mu
$$ 
has no complete solutions.   Thus $H$ must be nowhere vanishing on $M$, but this time \eqref{eqn:Hmu} yields
$$
(1/H)(t) = -\frac{1}{\mu} + c_1\sin(\sqrt{2|\mu|}\,t) + c_2 \cos(\sqrt{2|\mu|}\,t)
$$
and \eqref{ODE:h}
$$
f(t) = \mu+\frac{c_3}{-\frac{1}{\mu} + c_1\sin(\sqrt{2|\mu|}\,t) + c_2 \cos(\sqrt{2|\mu|}\,t)}\cdot
$$
But both of these are well behaved for appropriate constants, so there is no contradiction.)  In any case, when $\mu > 0$, then a distinguished orthonormal frame appears, which plays a crucial role in the proof of our Theorem:

\begin{prop}
\label{prop:2}
Assume the hypotheses of Proposition \ref{prop:1}.  If $\mu > 0$ and $M$ is closed and simply connected, then there exists a global orthonormal frame $\{T,X,Y\}$ with respect to which the spin coefficients $\kappa, \rho, \beta, \sigma$ take the form
\beqa
\label{eqn:good}
\kappa = \rho = \beta = 0 \comma \sigma = -\sqrt{\frac{\mu}{2}}+ i\,\frac{\omega}{2}\cdot
\eeqa
In particular, $T$ is geodesic, divergence-free, and twist-free.
\end{prop}
\begin{proof}
By Proposition \ref{prop:1}, $T$ has geodesic flow, so that $\kappa  = 0$; by its Corollary, $T$ is also divergence-free, so that $\rho = - i\frac{\omega}{2}$ (recall \eqref{eqn:rho2}).  Now we show the existence of a local orthonormal frame $\{T,X,Y\}$ with respect to which the shear $\sigma$ takes the form \eqref{eqn:good}.  Begin by observing that when $\kappa = \text{div}\,T = 0$, \eqref{ODE:div} reduces to
$$
2|\sigma|^2 - \frac{\omega^2}{2} = \mu.
$$
This in turn implies that the endomorphism $D\colon T^{\perp} \lra T^{\perp}$, whose matrix is given by \eqref{eqn:matrix}, has the two distinct eigenvalues $\pm\sqrt{\mu/2}$ (note that $D$ is self-adjoint with respect to the induced (Riemannian) metric $g|_{T^{\perp}}$ on $T^{\perp}$ if and only if $\omega$ vanishes).  Therefore, consider the respective kernels of the two bundle endomorphisms $D \pm \sqrt{\mu/2} I\colon T^{\perp} \lra T^{\perp}$; if $M$ is simply connected, then these line bundles have smooth nowhere vanishing global sections $X,Y_1$,
$$
D(X) = \sqrt{\mu/2}\,X \comma D(Y_1) = -\sqrt{\mu/2}\,Y_1,
$$
which we can take to have unit length, and which are both spacelike because they are orthogonal to $T$. If necessary, modify $Y_1$ so that it is orthogonal to $X$, by defining
$$
Y \defeq -g(X,Y_1)X+Y_1
$$
and normalizing $Y$ to have unit length.  We now claim that the global orthonormal frame $\{T,X,Y\}$ has shear $\sigma$ given by \eqref{eqn:good}; indeed, substituting \eqref{eqn:matrix} into $D(X) = \sqrt{\mu/2}\,X$ yields
\beqa
\label{eqn:wnot0}
\begin{pmatrix}-\sigma_1\\ \sigma_2-\frac{\omega}{2}\end{pmatrix} = \sqrt{\mu/2}\begin{pmatrix}1\\0\end{pmatrix} \imp \sigma_1 = -\sqrt{\mu/2} \comma \sigma_2 = \frac{\omega}{2},
\eeqa
completing the proof.
\end{proof}

\section{Proof of Theorem}
\label{sec:proof}
Armed with Propositions \ref{prop:1}, \ref{prop:2}, and our Corollary above, we now prove our Theorem, starting with the $\lambda > 0$ case:

\begin{proof}[\textcolor{blue}{Proof of $\lambda > 0$ case of Theorem.}]
(\emph{Once the sign changes (in red) are accounted for, this proof and Proposition \ref{prop:2} above are identical to that of Theorem 1 in \cite{AMB}}.)  Suppose a closed Lorentzian 3-manifold $(M,g)$ exists satisfying \eqref{eqn:1}, with $\lambda > 0$.  By the Corollary to Proposition \ref{prop:1}, $\text{div}\,T = 0$.  Now assume that $M$ is simply connected; then by Proposition \ref{prop:2}, there exists a global orthonormal frame $\{T,X,Y\}$ satisfying 
\beqa
\label{eqn:good2}
\kappa = 0 \comma \rho = -i\frac{\omega}{2} \comma \sigma = -\sqrt{\frac{\lambda}{2}} + i\frac{\omega}{2}\cdot
\eeqa
Using this information, we now show the existence of a vector field $Z$ and a smooth function $\psi$ on $M$ satisfying
\beqa
\label{eqn:Zf}
Z(\psi) = f,
\eeqa
which is impossible, as $f$ never takes the value $0$ and $M$ is closed. This will be shown by inserting \eqref{eqn:good2} into \eqref{eqn:S2}, \eqref{eqn:S3}, and \eqref{eqn:S4}. Indeed, doing so yields, in order,
$$
\underbrace{\,T(\sigma)\,}_{0} \overset{\eqref{eqn:S2}}{=} 2\sigma \vep \imp \sigma \vep = 0 \imp \vep = 0,\nonumber
$$
where  $T(\sigma) = 0$ because $\text{div}\,T = 0$, hence $T(\omega^2) = 0$ in \eqref{ODE:twist}; next, because $\beta = \frac{1}{\sqrt{2}}\big(g(\cd{Y}{X},Y) -ig(\cd{X}{Y},X)\big) = \frac{1}{\sqrt{2}}(\text{div}\,X - i\,\text{div}\,Y)$ (since $\cd{T}{T} = 0$),
\beqa
\underbrace{\,\mm(\rho) - \mb(\sigma)\,}_{-\frac{i}{\sqrt{2}}X(\omega)} \overset{\eqref{eqn:S3}}{=} 2\sigma \bar{\beta} \imp \left\{\begin{array}{c}
\sqrt{2\lambda}\,\text{div}\,X + \omega\,\text{div}\,Y = 0,\\
\sqrt{2\lambda}\,\text{div}\,Y - \omega\,\text{div}\,X = X(\omega).
\end{array}\label{eqn:*}
\right.
\eeqa
Finally, \eqref{eqn:S4}, when simplified using $\vep = 0$, yields
\beqa
\underbrace{\,\mm(\bar{\beta}) + \mb(\beta)\,}_{X(\text{div}\,X)\,+\,Y(\text{div}\,Y)} &=& \underbrace{\,\textcolor{red}{-}|\sigma|^2\,\textcolor{red}{+}\, |\rho|^2\,}_{-\lambda/2} -\, 2|\beta|^2 \,\textcolor{red}{-}\underbrace{\,(\rho - \bar{\rho})\vep\,}_{0}\nonumber\\
&&\hspace{.7in}- \underbrace{\,{\rm Ric}(\mm,\mb)\,}_{f} \,\textcolor{red}{-}\, \frac{1}{2}\underbrace{\,{\rm Ric}(\kk,\kk)\,}_{-\lambda}\nonumber\\
&=& -(\text{div}\,X)^2 - (\text{div}\,Y)^2 - f.\label{eqn:almost}
\eeqa
This further simplifies,
\beqa
X(\text{div}\,X) &\overset{\eqref{eqn:*}}{=}& -\frac{1}{\sqrt{2\lambda}}X(\omega)\,\text{div}\,Y -\frac{\omega}{\sqrt{2\lambda}}\,X(\text{div}\,Y)\nonumber\\
&\overset{\eqref{eqn:*}}{=}& -(\text{div}\,Y)^2 + \underbrace{\,\frac{\omega}{\sqrt{2\lambda}}(\text{div}\,X)(\text{div}\,Y)\,}_{-(\text{div}\,X)^2} -\ \frac{\omega}{\sqrt{2\lambda}}\,X(\text{div}\,Y),\nonumber
\eeqa
so that \eqref{eqn:almost} reduces, finally, to
\beqa
\label{eqn:compact}
\left(\frac{\omega}{\sqrt{2\lambda}}X - Y\right)(\text{div}\,Y) = f.
\eeqa
With $Z := \frac{\omega}{\sqrt{2\lambda}}X - Y$ and $\psi := \text{div}\,Y$, this is precisely \eqref{eqn:Zf}. This proves the Theorem in the case when $M$ is simply connected.  If $M$ is not simply connected, then pass to its simply connected universal cover $(\widetilde{M},\tilde{g})$; it is locally isometric to $(M,g)$ via the projection $\pi\colon \widetilde{M} \lra M$, and therefore its Ricci tensor $\widetilde{\text{Ric}}$ will satisfy \eqref{eqn:1}, with $f \circ \pi$ in place of $f$, and with $\widetilde{T}$ the (complete) lift of $T$.  Repeating step-by-step our argument on $(\widetilde{M},\tilde{g})$, we arrive once again at \eqref{eqn:compact}. Although $\widetilde{M}$ need not be compact, a contradiction is still obtained because $f \circ \pi$ is bounded away from zero, because $\text{div}\,Y$ is also bounded (since $d\pi(Y)$ is well defined up to sign, $|\text{div}\,d\pi(Y)|$ is continuous on $M$), and because $Z$ is complete on $\widetilde{M}$.  This completes the proof.
\end{proof}

\begin{proof}[\textcolor{blue}{Proof of $\lambda = 0$ case of Theorem.}]
(\emph{This proof parallels that of Theorem 3 in \cite{AMB}, but generalizes it: due to our Corollary above, the function $f$ is not assumed to be constant, as it was in \cite{AMB}; furthermore, \eqref{eqn:hh} below is a necessary step that was not required in \cite{AMB}.}) When $\lambda  = 0$, 
\beqa
\label{eqn:Tpart1}
|\sigma|^2 - \frac{\omega^2}{4}  = 0 \imp D~\text{has eigenvalues}~0,0.
\eeqa
As there are no longer two distinct eigenvalues, we cannot call upon Proposition \ref{prop:2}; instead, we prove that if \eqref{eqn:Tpart1} holds then $T$ must be parallel.  Doing so will then allow us to draw two conclusions: first, that $(M,g)$ must be geodesically complete, by \cite{SanRom}; second, that the universal cover of $(M,g)$ must be isometric to $(\RR \times \widetilde{N},-dt^2\oplus \tilde{h})$ for some Riemannian 2-manifold $(\widetilde{N},\tilde{h})$, by the de Rham Decomposition Theorem for Lorentzian manifolds \cite{Wu}. To begin with, observe that $T$ being parallel is equivalent to the  condition
$$
\kappa =  \rho = \sigma = 0.
$$
By Proposition \ref{prop:1} and its Corollary, $\cd{T}{T} = \text{div}\,T = 0$, so that we need only show that $\omega^2 = 0$; we'll do this by showing that the open set
$$
U = \{p \in M : \omega^2(p) \neq 0\}
$$
is empty. Assume for the moment that $U$ is simply connected. Then over $U$, $D$ has constant rank 1 (recall \eqref{eqn:matrix}), so that its kernel is a line bundle over $U$; as the latter is simply connected, this kernel has a nowhere vanishing section $X$ on $U$. Now let $\{T,X,Y\}$ be an orthonormal frame, with $Y$ perhaps only locally defined in $U$. Then the analogue of \eqref{eqn:wnot0} is now
$$
\begin{pmatrix}-\sigma_1\\ \sigma_2-\frac{\omega}{2}\end{pmatrix} = \begin{pmatrix}0\\0\end{pmatrix} \imp \sigma_1 = 0 \comma \sigma_2 = \frac{\omega}{2},
$$
and thus the analogue of \eqref{eqn:good} is $\rho = -i\frac{\omega}{2}$ and $\sigma = i\frac{\omega}{2}.$  Proceeding as in the proof of the $\lambda > 0$ case above, \eqref{eqn:*} becomes
\beqa
\label{eqn:**}
\omega (\text{div}\,Y) = 0 \comma X(\omega) = -\omega\,(\text{div}\,X).
\eeqa
Now $\text{div}\,Y = 0$ because $\omega$ is nowhere vanishing in $U$, in which case \eqref{eqn:almost} reduces to
\beqa
\label{eqn:divX}
X(\text{div}\,X) = -(\text{div}\,X)^2  - f.
\eeqa
If the flow of $X$ was complete in $U$, then we would obtain a contradiction because \eqref{eqn:divX} has no complete solutions when $f > 0$.  Our task is therefore done if we can prove that the flow of $X$ is complete. Thus, let $\gamma\colon [0,b) \lra U$ be an integral curve of $X$ that is maximally extended to the right, and suppose that $b < \infty$ (the case $(-b,0]$ will follow from this one by letting $X \to -X$, which leaves \eqref{eqn:**} and \eqref{eqn:divX} unaltered).  To begin with, there is a sequence $t_n \to b$ such that $\{\gamma(t_n)\}$ converges to some $q \in M\backslash U$ (if $q$ were in $U$, then the integral curve $\gamma$ would be extendible, contradicting our assumption that it was maximally extended; see, e.g., \cite[Lemma 56]{o1983}). Let us give a proof of this that will also suffice should we need to pass to the universal cover of $M$ below: consider the Riemannian metric $h$ on $M$ defined by
\beqa
\label{eqn:hh}
h \defeq g + 2T^{\flat}\otimes T^{\flat}. 
\eeqa
Since $M$ is closed, $h$ is complete; as $X$ has $h$-unit length,
$$
d_h(\gamma(0),\gamma(t)) \leq L_h(\gamma|_{[0,t]}) = t,
$$
where $d_h$ is the Riemannian distance associated to $h$.  This implies that
$$
\gamma([0,b)) \subseteq \{p \in M : d_h(\gamma(0),p) \leq b\}.
$$
By the completeness of $(M,h)$, the latter set is compact, hence any sequence $\{\gamma(t_n)\}$ with $t_n \to b$ has a convergent subsequence; cf. \cite[Proposition 3.4]{candela}.  Now, set $\theta(t) \defeq (\text{div}\,X \circ \gamma)(t)$ and $\omega^2(t) \defeq (\omega^2 \circ \gamma)(t)$; in particular, observe that $\lim_{n\to \infty} \omega^2(t_n) = 0$ because $q\notin U$.  By \eqref{eqn:**},
$$
\omega^2(t) = \omega_0^2\,e^{-2\!\int_0^t \theta(s)\,ds}\hspace{.2in}\text{for all}~t \in [0,b).
$$
By \eqref{eqn:divX}, $\theta(t)$ is nonincreasing ($f > 0$), in which case
$
-2\!\int_0^t \theta(s)\,ds \geq -2\theta_0t \geq -2\theta_0b$ for all $t \in [0,b).$
But then
$$
\omega^2(t) \geq w_0^2e^{-2\theta_0b} > 0 \imp \lim_{n \to \infty} w^2(t_n) > 0,
$$
a contradiction. Thus we must have $b = \infty$; this proves that $U$, if simply connected, must be empty.  If $U$ is not simply connected, then pass to the universal covers of $(U,g|_U) \subset (M,g)$ and repeat the argument (with $\omega^2 \circ \pi$, and noting that the lift of the Riemannian metric \eqref{eqn:hh} will be complete), noting that any integral curve of $T$ starting in $U$ stays in $U$, because $T(\omega^2)  = 0$ via \eqref{ODE:twist}.  This completes the proof that the universal cover of $(M,g)$ is isometric to $(\RR \times \widetilde{N},-dt^2\oplus \tilde{h})$, in which case $(M,g)$ itself is isometric to $(\mathbb{S}^1 \times N,-dt^2 \oplus h)$ with $(N,h)$ a Riemannian 2-manifold. This establishes the $\lambda = 0$ case of the Theorem.
\end{proof}

\section*{Acknowledgements}
We kindly thank the anonymous referee for a careful reading of this manuscript, and Miguel S\'anchez for very helpful discussions.

\bibliographystyle{alpha}
\bibliography{JMAA_final_version}

\begin{thebibliography}{AMT19}

\bibitem[AMT19]{AMB}
Amir~Babak Aazami and Charles~M. Melby-Thompson.
\newblock On the principal {R}icci curvatures of a {R}iemannian 3-manifold.
\newblock {\em Advances in Geometry}, 19(2):251--262, 2019.

\bibitem[BS18]{bettiol2018}
Renato Bettiol and Benjamin Schmidt.
\newblock Three-manifolds with many flat planes.
\newblock {\em Transactions of the American Mathematical Society},
  370(1):669--693, 2018.

\bibitem[Bue97a]{bueken2}
Peter Bueken.
\newblock On curvature homogeneous three-dimensional {L}orentzian manifolds.
\newblock {\em Journal of Geometry and Physics}, 22(4):349--362, 1997.

\bibitem[Bue97b]{bueken}
Peter Bueken.
\newblock Three-dimensional {L}orentzian manifolds with constant principal
  {R}icci curvatures $\rho_1 = \rho_2 \ne \rho_3$.
\newblock {\em Journal of Mathematical Physics}, 38(2):1000--1013, 1997.

\bibitem[Car89]{carriere}
Yves Carri{\`e}re.
\newblock Autour de la conjecture de {L}. {M}arkus sur les vari{\'e}t{\'e}s
  affines.
\newblock {\em Inventiones Mathematicae}, 95(3):615--628, 1989.

\bibitem[CK09]{calvaruso}
Giovanni Calvaruso and Oldrich Kowalski.
\newblock On the {R}icci operator of locally homogeneous {L}orentzian
  3-manifolds.
\newblock {\em Open Mathematics}, 7(1):124--139, 2009.

\bibitem[CM62]{CM}
Eugenio Calabi and Lawrence Markus.
\newblock Relativistic space forms.
\newblock {\em Annals of Mathematics}, pages 63--76, 1962.

\bibitem[CS08]{candela}
Anna~Maria Candela and Miguel S{\'a}nchez.
\newblock {\em Geodesics in semi-{R}iemannian manifolds: geometric properties
  and variational tools}, volume~4.
\newblock European Mathematical Society Z{\"u}rich, 2008.

\bibitem[Gol85]{Goldman}
William~M. Goldman.
\newblock Nonstandard {L}orentz space forms.
\newblock {\em Journal of Differential Geometry}, 21(2):301--308, 1985.

\bibitem[HMP87]{hall87}
G.S. Hall, T.~Morgan, and Z.~Perj{\'e}s.
\newblock Three-dimensional space-times.
\newblock {\em General relativity and gravitation}, 19(11):1137--1147, 1987.

\bibitem[Kli96]{Klingler}
Bruno Klingler.
\newblock Compl{\'e}tude des vari{\'e}t{\'e}s lorentziennes {\`a} courbure
  constante.
\newblock {\em Mathematische Annalen}, 306(2):353--370, 1996.

\bibitem[KR85]{kulkarni}
Ravi~S. Kulkarni and Frank Raymond.
\newblock 3-dimensional {L}orentz space-forms and {S}eifert fiber spaces.
\newblock {\em Journal of Differential Geometry}, 21(2):231--268, 1985.

\bibitem[Lim10]{limoncu}
Murat Limoncu.
\newblock Modifications of the {R}icci tensor and applications.
\newblock {\em Archiv der Mathematik}, 95(2):191--199, 2010.

\bibitem[Lun15]{lundberg}
David Lundberg.
\newblock On the non-existence of compact {L}orentzian manifolds with constant
  positive curvature.
\newblock {\em Master's Thesis, Lund University}, 2015.

\bibitem[NP62]{newpen62}
Ezra Newman and Roger Penrose.
\newblock An approach to gravitational radiation by a method of spin
  coefficients.
\newblock {\em Journal of Mathematical Physics}, 3(3):566--578, 1962.

\bibitem[NTC15]{NTC}
Pawe{\l} Nurowski and Arman Taghavi-Chabert.
\newblock A {G}oldberg--{S}achs theorem in dimension three.
\newblock {\em Classical and Quantum Gravity}, 32(11):115009, 2015.

\bibitem[Ole14]{olea}
Benjam{\'\i}n Olea.
\newblock Canonical variation of a {L}orentzian metric.
\newblock {\em Journal of Mathematical Analysis and Applications},
  419(1):156--171, 2014.

\bibitem[O'N83]{o1983}
Barrett O'Neill.
\newblock {\em Semi—-{R}iemannian {G}eometry with {A}pplications to
  {R}elativity}, volume 103.
\newblock Academic press, 1983.

\bibitem[RS95]{SanRom}
Alfonso Romero and Miguel S{\'a}nchez.
\newblock Completeness of compact {L}orentz manifolds admitting a timelike
  conformal {K}illing vector field.
\newblock {\em Proceedings of the American Mathematical Society},
  123(9):2831--2833, 1995.

\bibitem[RS96]{sanchez_bochner}
Alfonso Romero and Miguel S{\'a}nchez.
\newblock An integral inequality on compact {L}orentz manifolds, and its
  applications.
\newblock {\em Bulletin of the London Mathematical Society}, 28(5):509--513,
  1996.

\bibitem[RS98]{sanchez_bochner2}
Alfonso Romero and Miguel S{\'a}nchez.
\newblock Bochner's technique on {L}orentzian manifolds and infinitesimal
  conformal symmetries.
\newblock {\em Pacific Journal of Mathematics}, 186(1):141--148, 1998.

\bibitem[SW12]{Wu}
Rainer~Kurt Sachs and H-H Wu.
\newblock {\em General {R}elativity for mathematicians}, volume~48.
\newblock Springer Science \& Business Media, 2012.

\bibitem[SW14]{schmidt2014}
Benjamin Schmidt and Jon Wolfson.
\newblock Three-manifolds with constant vector curvature.
\newblock {\em Indiana University Mathematics Journal}, 63(6):1757--1783, 2014.

\end{thebibliography}
\end{document}